\documentclass[a4paper,11pt,oneside,bookmarks]{article}
\usepackage[utf8x]{inputenc}
\usepackage{graphicx}
\usepackage{wrapfig}
\usepackage[pdftex,bookmarks]{hyperref}
\hypersetup{
pdftitle={On characterization of locally $L^0$-convex topologies induced by a family of $L^0$-seminorms},
pdfauthor={José M. Zapata}
}

\usepackage{fancyhdr}
\usepackage[english]{babel}
\usepackage{amsmath}
\usepackage{amsthm}
\usepackage{amsfonts}

\usepackage{amscd}

\usepackage{epigraph}

\usepackage{amssymb}

\newcommand{\N}{\mathbb{N}}
\newcommand{\R}{\mathbb{R}}
\newcommand{\E}{\mathbb{E}}
\newcommand{\PP}{\mathbb{P}}

\DeclareMathOperator*{\esssup}{ess.sup\,}
\DeclareMathOperator*{\essinf}{ess.inf\,}

\DeclareMathOperator*{\interior}{int\, }
\DeclareMathOperator{\sgn}{sgn}

\newtheorem{defn}{Definition}[section]

\newtheorem{prop}{Proposition}[section]

\newtheorem{thm}{Theorem}[section]
\newtheorem{cor}{corollary}[section]
\newtheorem{example}{Example}[section]

\providecommand{\keywords}[1]{\textbf{\textit{Keywords:}} #1}

\begin{document}


\title{A random version of Mazur's lemma}

\author{José M. Zapata \thanks{Universidad de Murcia, Dpto. Matemáticas, 30100 Espinardo, Murcia, Spain, e-mail: jmzg1@um.es}}

\date{\today}
\maketitle


\begin{abstract}
The purpose of this paper is to generalize the classical Mazur's lemma from the classical convex analysis to the framework of locally $L^0$-convex modules. In this version an extra condition of countable concatenation is included. We provide a counterexample showing that this condition cannot be removed. 
\end{abstract}

 

\section*{Introduction}

%
%
%

In recents years, works like \cite{key-2,key-3,key-5,key-6,key-8,key-9} have highlighted that the appropriate theoretical framework in which embed the theory of conditional risk measures is the theory of locally $L^0$-convex modules.

As \cite{key-3} propose, to carry out this study it is necessary to bring tools from classical convex analysis fitting them to this new framework. Ultimately, to create a randomized generalization of convex analysis. 

Due to difficulties deriving from working with scalars into the ring $L^0$ instead of $\R$, some obstacles must be overcome, namely, as shown in \cite{key-3} and \cite{key-8}, mainly the fact that not all the non-zero elements possess a multiplicative inverse, i. e, $L^0$ is not a field, and that $L^0$ is not endowed with a total order.

Thus, some theorems from convex analysis will remain valid in the $L^0$-convex modules, but others will require additional conditions. So in \cite{key-3} and \cite{key-8}, some notions known as countable concatenation properties are addressed.

An important tool of classical convex analysis is the Mazur's lemma, which shows that for any weakly convergent sequence in a Banach space there is a sequence of convex combinations of its members that converges strongly to the same limit. This allows in some situations changing weak topology by strong topology and vice versa working with normed spaces.

Then, the purpose of this article is to show a randomize version for $L^0$-normed modules, finding that an extracondition of countable concatenation is needed.  

This paper is structured as follow:  We give a first section of preliminaries. In the second section we study some properties of the gauge function for $L^0$-modules. And finally, the third section is devoted to the Mazur's lemma for $L^0$-modules, proving the main result and a  counterexample showing that the extra condition of countable concatenation cannot be removed.
 
\section{Preliminaries}

Given a probability space $\left(\Omega,\mathcal{F},\PP\right)$, which will be fixed for the rest of this paper, we consider $L^{0} \left(\Omega,\mathcal{F},\PP\right)$, the set of equivalence classes of  real valued $\mathcal{F}$-measurable random variables, which will be denoted simply as $L^{0}$.

It is known that the triple $\left(L^{0},+,\cdot\right)$ endowed with the partial order of the almost sure dominance is a lattice ordered ring.

We write ``$X\geq Y$`` if $\PP\left( X\geq Y \right)=1$.
Likewise, we write ``$X>Y$'', if $\PP\left( X> Y \right)=1$. 

And, given $A\in \mathcal{F}$, we write $X>Y$ (respectively,  $X \geq Y$) on $A$, if $\PP\left(X>Y \mid A\right)=1$ (respectively , if $\PP\left(X \geq Y \mid  A \right)=1$).

We also define 

\[
L_{+}^{0}:=\left\{ Y\in L^{0};\: Y\geq 0\right\} 
\]
and
\[
L_{++}^{0}:=\left\{ Y\in L^{0};\: Y>0  \right\}. 
\]

And denote by $\bar{L^{0}}$, the set of equivalence classes of  $\mathcal{F}$-measurable random variables taking values in $\bar{\R}=\R\cup\{\pm\infty\}$, and extend the partial order of the almost sure dominance to $\bar{L^{0}}$.

In A.5 of \cite{key-4} is proved the proposition below:

\begin{prop}
\label{prop: esssup}
Let $\phi$  be a subset of $L^{0}$, then
\begin{enumerate}
\item There exists $Y^{*}\in\bar{L^{0}}$ such that $Y^{*}\geq Y$ for all 
$Y\in\phi,$ and such that any other $Y'$ satisfying the same, verifies $Y'\geq Y^{*}.$
\item Suppose that  $\phi$ is directed upwards. Then there exists a increasing sequence  $Y_{1}\leq Y_{2}\leq...$ in $\phi,$ such that 
 $Y_{n}$ converges to $Y^{*}$ almost surely.
\end{enumerate}
\end{prop}.

\begin{defn}
Under the conditions of the previous proposition, 
$Y^{*}$ is called essential supremum of $\phi$, and we write

\[
\esssup \phi=\underset{Y\in\phi}{\esssup Y}:=Y^{*}.
\]

The essential infimum of $\phi$ is defined as

\[
\essinf \phi=\underset{Y\in\phi}{\essinf\, Y}:=\underset{Y\in\phi}{-\esssup \left(-Y\right)}.
\]
\end{defn}


The order of the almost sure dominance also lets us define a topology on $L^{0}$. Let 

\[
B_{\varepsilon}:=\left\{ Y\in L^{0};\:\left|Y\right|\leq\varepsilon\right\} 
\]

the ball of radius $\varepsilon\in L_{++}^{0}$ centered at $0\in L^{0}$.
Then, for all $Y\in L^{0}$, $\mathcal{\mathcal{U}}_{Y}:=\left\{ Y+B_{\varepsilon};\:\varepsilon\in L_{++}^{0}\right\} $ is a neighborhood base of $Y$. Thus, it can be defined a topology on $L^{0}$ that it will be known as the topology induced by $\left|\cdot\right|$ and $L^{0}$ endowed with this topology will be denoted by $L^{0}\left[\left|\cdot\right|\right]$.

Now, we will give the central concepts of the theory of locally $L^0$-convex modules.

\begin{defn}
A topological $L^{0}$-module $E\left[\tau\right]$ is a $L^{0}$-module $E$ endowed with a topology $\tau$ such that 
\begin{enumerate}
\item $E\left[\tau\right]\times E\left[\tau\right]\longrightarrow E\left[\tau\right],\left(X,X'\right)\mapsto X+X'$
and
\item $L^{0}\left[\left|\cdot\right|\right]\times E\left[\tau\right]\longrightarrow E\left[\tau\right],\left(Y,X\right)\mapsto {Y}X$
\end{enumerate}
are continuous with the corresponding product topologies.
\end{defn}

\begin{defn}
A topology $\tau$ on a $L^{0}$-module $E$ is locally $L^{0}$-convex  
 if there is a neighborhood base of $0\in{E}$ $\mathcal{U}$ such that each $U\in\mathcal{U}$ is
\begin{enumerate}
\item $L^{0}$-convex, i.e. ${Y}X_{1}+{(1-Y)}X_{2}\in U$ for all $X_{1},X_{2}\in U$
and $Y\in L^{0}$ with $0\leq Y\leq1,$
\item $L^{0}$-absorbent, i.e. for all $X\in E$ there is a $Y\in L_{++}^{0}$
such that $X\in {Y}U,$
\item $L^{0}$-balanced, i.e. ${Y}X\in U$ for all $X\in U$ and $Y\in L^{0}$
with $\left|Y\right|\leq1.$

And, as discussed in \cite{key-13}, we add an extracondition

\item $U$ has the relative countable concatenation property.

In this case, $E\left[\tau\right]$ is called locally $L^0$-convex module. 
\end{enumerate}
\end{defn}

The notion of having the relative countable concatenation property will be recalled later (see Definition \ref{def: cc}).

\begin{defn}
A function $\left\Vert \cdot\right\Vert :E\rightarrow L_{+}^{0}$
is a $L^{0}$-seminorm on $E$ if:
\begin{enumerate}
\item $\left\Vert {Y}X\right\Vert =\left|Y\right|\left\Vert X\right\Vert $
for all $Y\in L^{0}$ y $X\in E.$
\item $\left\Vert X_{1}+X_{2}\right\Vert \leq\left\Vert X_{1}\right\Vert +\left\Vert X_{2}\right\Vert ,$
for all $X_{1},X_{2}\in E.$

If, moreover

\item $\left\Vert X\right\Vert =0$ implies $X=0,$
\end{enumerate}
then $\left\Vert \cdot\right\Vert $ is a $L^{0}$-norm on
$E$.

\end{defn}
\begin{defn}
Let $\mathcal{P}$ be a family of $L^0$-seminorms on a $L^{0}$-module
$E$. Suppose that $Q\subset\mathcal{P}$ is finite and $\varepsilon\in L_{++}^{0},$
we define 

\[
U_{Q,\varepsilon}:=\left\{ X\in E;\:\underset{\left\Vert .\right\Vert \in Q}{\sup}\left\Vert X\right\Vert \leq\varepsilon\right\} .
\]
Then for all $X\in E$, $\mathcal{\mathcal{U}}_{Q,X}:=\left\{ X+U_{\varepsilon};\:\varepsilon\in L_{++}^{0},\: Q\subset \mathcal{P}\: {finite} \right\} $
is a neighbourhood base of $X$. Thereby, we define a topology
on $E$, which it will be known as the topology induced by $\mathcal{P}$
and $E$ endowed with this topology will be denoted by $E\left[\mathcal{P}\right]$.

In addition, it is proved by the lemma 2.16 of \cite{key-3} that $E\left[\mathcal{P}\right]$ is a locally $L^0$-convex module.

Furthermore, according to \cite{key-13} a topological $L^0$-module $E\left[\tau\right]$ is a locally $L^0$-convex module if, and only if, $\tau$ is induced by a family of $L^0$-seminorms.
\end{defn}

\begin{defn}
Given a topological $L^{0}$-module $E\left[\tau\right]$, we denote by $E[\tau]^*$, or simply by $E^*$, the $L^0$-module of continuous $L^0$-linear functions $\mu : E \rightarrow L^0$.

We define
\[
\begin{array}{c}
\left\langle \cdot, \cdot\right\rangle : E \times E^*\longrightarrow L^0\\
\left\langle X, X^* \right\rangle:=X^*(X).
\end{array}
\]

For each $X^* \in E^*$ it holds that 
\[
\begin{array}{c}
q_{X^*}:E\rightarrow L^0_+ \\
q_{X^*}(X):=\left|\left\langle X, X^* \right\rangle\right|
\end{array}
\]
 is a $L^0$-seminorm. 

Now, consider the topology $\sigma(E,E^*)$ induced by the family of $L^0$-seminorms
\[
\left\{q_{X^*} ; \: X^* \in E^* \right\}.
\]
Then, $\sigma(E,E^*)$ is a locally $L^0$-convex topology, which is called the weak topology of $E$.

Likewise, for each $X\in E$ it holds that
\[
\begin{array}{c}
q_{X}:E^*\rightarrow L^0_+ \\
q_{X}(X^*):=\left|\left\langle X, X^* \right\rangle\right|
\end{array}
\]
 is a $L^0$-seminorm.

And we have the  $L^0$-convex topology $\sigma(E^*,E)$ induced by the family of $L^0$-seminorms
\[
\left\{q_{X} ; \: X \in E \right\},
\]
which is called the weak-$*$ topology of $E$.  
\end{defn}
\section{The gauge function and the countable concatenation closure.}

Let us write the notion of gauge function introduced in \cite{key-3}:  

\begin{defn}
Let $E$ be a $L^0$-module. The gauge function $p_{K}:E\rightarrow\bar{L}_{+}^{0}$ of
a set $K\subset E$ is defined by 

\[
p_{K}\left(X\right):=\essinf\left\{ Y\in L_{+}^{0};\: X\in {Y}K\right\} .
\]

In addition,  if $K$ is $L^{0}$-convex, $L^{0}$-absorbent and $L^0$-balanced, then $p_{K}$ is a $L^0$-seminorm (see \cite{key-3}, Proposition 2.23).

\end{defn}

Now we will give the notion of having the relative countable concatenation property, which is based on the notion of countable concatenation property given in \cite{key-8}. In \cite{key-3} the authors work with two others notions of countable concatenation property, one for the topology and other for the family of $L^0$-seminorms, although both properties turn out to be the same.
 
The notion introduced in \cite{key-8}, and the one given below, are related to the $L^0$-module itself rather than the topology.  

\begin{defn}
\label{def: cc}
Let $E$ be a $L^{0}$-module and $K\subset E$ a subset, and denote by $\Pi\left(\Omega,\mathcal{F}\right)$ the set of countable partitions on $\Omega$ to $\mathcal{F}$.
\begin{itemize}
\item  Given a sequence $\left\{ X_{n}\right\}_{n\in \mathbb{N}}$
in $E$ and a partition $\left\{ A_{n}\right\}_{n\in \mathbb{N}}\in\Pi(\Omega,\mathcal{F})$, 
we define the set of countable concatenations of $\{X_n\}_n$ and $\{A_n\}_n$ as
\[
cc\left(\{A_{n}\}_n,\{X_{n}\}_n\right):=\{X\in E;\: 1_{A_{n}}X_{n}=1_{A_{n}}X \textnormal{ for each } n\in \mathbb{N}\}.
\]


\item   $K$ is said to have the \emph{relative countable concatenation property}, if for each sequence $\left\{ X_{n}\right\}_{n}$
 in $K$ and each partition
$\left\{ A_{n}\right\}_{n}\in\Pi(\Omega,\mathcal{F})$ 
it holds  
\[
cc\left(\{A_{n}\}_n,\{X_{n}\}_n\right)\subset K.
\]

\item  We call the countable concatenation closure of $K$ the set defined below 
\[
\overline{K}^{cc}:=\bigcup cc\left(\{A_{n}\}_n,\{X_{n}\}_n\right)
\]

where $\left\{ A_{n}\right\}_{n}$ runs through $\Pi(\Omega,\mathcal{F})$ and $\left\{ X_{n}\right\}_{n}$
runs though the sequences in $K$. Or in another way written
\[
\overline{K}^{cc}=\{X\in E; \: \exists \{A_n\}_{n\in\N} \in \Pi(\Omega,\mathcal{F}) \textnormal{ with } 1_{A_n}X\in 1_{A_n}K \textnormal{ for all }n  \}. 
\]
\end{itemize}
\end{defn}

It is clear that $K$ has the relative countable concatenation property if, and only if, $\overline{K}^{cc}=K$.


%
%
%
%
%
%
 %

Since $L^0$ is not a totally ordered set, we need to take advantage of the notion of countable concatenation. We have the following result:

\begin{prop}
\label{prop: cccdownwardsDirected}
Let $C\subset L^0$ be bounded below (resp. above) and stable under countable concatenations,  then for each $\varepsilon\in L^0_{++}$ there exists $Y_\varepsilon \in C$ such that 
\[
\essinf C \leq Y_\varepsilon < \essinf C + \varepsilon
\]
 (resp., $\esssup C \geq Y_\varepsilon > \esssup C - \varepsilon$) 

In particular, given a $L^0$-module $E$ and $K\subset E$, which is $L^0$-convex, $L^0$-absorbent and stable under countable concatenations, we have that, for $\varepsilon\in L^0_{++}$, there exists $Y_\varepsilon \in L^0_{++}$ with $X \in {Y_\varepsilon}K$ such that $p_K(X)\leq Y_\varepsilon < p_K(X)+\varepsilon$. 
\end{prop}  
\begin{proof}

Firstly, let us see that $C$ is downwards directed. Indeed, given $Y,Y'\in C$, define $A:=(Y<Y')$. Then, since $C$ is stable under countable concatenations, $1_A Y + 1_{A^c} Y'=Y\wedge Y' \in C$. 

Therefore,  for $\varepsilon \in L^0_{++}$, there exists a decreasing sequence $\left\{Y_k\right\}_{k}$ in $C$ converging  to $\essinf C$ almost surely. 

Let us consider the sequence of sets
\[
A_0:=\emptyset \textnormal{ and } A_k:=(Y_k<\essinf C +\varepsilon)-A_{k-1} \textnormal{ for }k>0. 
\]
Then $\left\{A_k\right\}_{k\geq 0}\in\Pi(\Omega,\mathcal{F}^+)$, and we can define $Y_\varepsilon:=\sum_{k\geq 0}{1_{A_k}}Y_k$. Given that $C$ is stable under countable concatenations, it follows that $Y_\varepsilon \in C$.

For the second part, it suffices to see that if $K$ is stable under countable concatenations then $\left\{Y\in L^0_{++}; \: X\in {Y}K  \right\}$ is stable under countable concatenations as well. Indeed, given $\left\{A_k\right\}_{k}\in\Pi(\Omega,\mathcal{F})$ and $\left\{Y_k\right\}_{k} \subset L^0_{++}$ such that $X\in {Y_k}K$ for each $k\in\N$, let us take $Y:={\sum}_{n\in\mathbb{N}}Y_{n}1_{A_{n}}\in L_{++}^{0}$. Then we have that $X/Y\in cc\left(\{A_{k}\}_k,\left\{X/Y_{k}\right\}_k\right)$ and $X/Y_{k}\in K$.

Since $K$ is stable under countable concatenations, we conclude that  $X/Y\in K$, and the proof is complete.  
\end{proof}

Now, we have the following proposition

\begin{prop}
\label{closureGauge}
Let $E\left[\tau\right]$ a topological $L^0$-module and 
$C\subset E$ a $L^{0}$-convex and $L^{0}$-absorbent subset.  Then the following are equivalent 
\begin{enumerate}
	\item $p_C:E\rightarrow L^0$ is continuous. 
	\item $0\in \interior{C}$
\end{enumerate}

In this case, if in addition $C$ has the relative countable concatenation property
\[
\overline{C}=\left\{X \in E ; \: p_C(X)\leq 1  \right\}.
\]
\end{prop}

\begin{proof}

The proof is exactly the same as the real case.

For the equality.

``$\subseteq$'': It is obtained from the continuity.

``$\supseteq$'': Let $X\in E$ be satisfying $p_C(X)\leq 1$. By proposition \ref{prop: cccdownwardsDirected}, we have that for every $\varepsilon \in L^0_{++}$ there exists $Y_\varepsilon \in L^0_{++}$ such that
\[
1 \leq Y_\varepsilon < 1+\varepsilon,
\]
with $X \in Y_\varepsilon C$.

Then, $\left\{X/Y_\varepsilon \right\}_{\varepsilon \in L^0_{++}}$ is a net in $C$ converging to $X$. Thus $X\in \overline{C}$. And the proof is complete.
\end{proof}

Let us see an example showing that for the equality proved in the last proposition it is necessary to take $C$ with the relative countable concatenation property. 
\begin{example}
Given  $\Omega=(0,1)$, $\mathcal{E}=\mathcal{B}(\Omega)$ the $\sigma$-algebra of Borel, $A_{n}=[\frac{1}{2^n},\frac{1}{2^{n-1}})$ with $n\in\mathbb{N}$, $\PP$ the Lebesgue measure and $E:=L^0(\mathcal{E})$ endowed with $|\cdot|$.

We define the set 

\[
U:=\left\{ Y\in L^{0};\:\exists\, I\subset\mathbb{N}\textnormal{ finite, }\left|Y1_{A_{i}}\right|\leq 1\:\forall\, i\in\mathbb{N}-I\right\}.
\]

Then, it is easily shown that $U$ is
$L^{0}$-convex, $L^{0}$-absorbent and $\overline{U}=U$.

Nevertheless, it can be proved that $p_{U}\left(X\right)=0$ for all $X\in{L^0}$, and therefore
\[
\left\{ X\in E;\: p_{U}(X)\leq 1\right\} = E.
\]

\end{example}

\section{A random version of the Mazur's lemma}

Finally, let us turn to prove a version for $L^0$-modules of the classical Mazur's lemma.

We need a new notion:

\begin{defn}
Let $E$ be a $L^0$-module, we say that the sum of $E$ preserves the relative countable concatenation property, if for every $L,M$ subsets of $E$ with the relative countable concatenation property it holds that the sum of both $L+M$ has the relative countable concatenation property. 
\end{defn} 

%

\begin{thm}
\label{thm: Mazur}
[Randomized version of the Mazur lemma]
Let $(E,\left\Vert \cdot\right\Vert)$ be a $L^0$-normed module whose sum preserves the relative countable concatenation property, and let $\{X_\gamma\}_{\gamma\in \Gamma}$ be a net in $E$, which converges weakly to $X\in E$. Then, for any $\varepsilon\in L^0_{++}$, there exists
\[
Z_{\varepsilon} \in \overline{co}^{cc}_{L^0}\{X_\gamma ; \:\gamma\in \Gamma \} \textnormal{ such that }\left\Vert X - Z_\varepsilon \right\Vert \leq \varepsilon.
\] 

\end{thm}



\begin{proof}
Define $M_1 := \overline{co}^{cc}_{L^0} \{X_\gamma ; \: \gamma\in \Gamma \}.$ We may assume that $0\in M_1$, by replacing $X$ by $X-X_{\gamma_0}$ and $X_\gamma$ by $X_\gamma-X_{\gamma_0}$ for some $\gamma_0\in \Gamma$ fixed and all $\gamma\in\Gamma$.

By way of contradiction, suppose that for every $Z \in M_1$ there exists $A_Z \in \mathcal{F}^+$  such that $\left\Vert X - Z \right\Vert > \varepsilon$ on $A_Z$.

Denote $B_\frac{\varepsilon}{2}:=\{ X\in E; \Vert X \Vert\leq \frac{\varepsilon}{2}\}$, and define $M:=\underset{Z\in M_1}\bigcup{Z + B_\frac{\varepsilon}{2}}.$

Then $M$ is a $L^0$-convex $L^0$-absorbent neighbourhood of $0\in E$, which has the relative countable concatenation property (this is because $M_1$ and $B_\frac{\varepsilon}{2}$ has the relative c. c. property and the sum of $E$ preserves the relative c. c. property). Besides, for every $Z \in M$ there exists $C_Z \in \mathcal{F}^+$ with $\Vert X-Z\Vert\geq\frac{\varepsilon}{2}$ on $C_Z$. So that $X \notin \overline{M}$.


Thus, by Proposition \ref{closureGauge} we have that there exists $C\in \mathcal{F}^+$ such that
\begin{equation}
\label{ineqII}
\begin{array}{cc}
p_M (X) > 1 & \textnormal{ on } C,
\end{array}
\end{equation}

where $p_M$ is the guage function of $M$.

Further, given $Y,Y' \in L^0$ with $1_C Y X =1_C Y' X$, it holds that $Y=Y'$ on $C$. 

Indeed, define $A=(Y-Y'\geq 0)$
\[
1_C |Y-Y'| p_M (X)\leq p_M (1_C |Y-Y'|X)=p_M ((1_{A}-1_{A^c})1_C (Y-Y')X)=p_M(0)=0.
\]
In view of \ref{ineqII}, we conclude that $Y=Y'$ on $C$.  

Then, we can define the following $L^0$-linear application 
\[
\begin{array}{cc}
\mu_0 : & span_{L^0} \{X\}\longrightarrow L^0\\
& \mu_0 (Y X):=Y 1_{C} p_M(X).
\end{array}
\]

In addition, we have that
\[
\begin{array}{cc}
\mu_0 (Z) \leq p_M (Z) & \textnormal{ for all } Z\in span_{L^0}\{X\}.
\end{array}
\] 

Thus, by the Hahn-Banach extension theorem for $L^0$-modules (\cite{key-3}, theorem 2.14), there exists a $L^0$-linear extension $\mu$ of $\mu_0$ defined on $E$ such that 
\[
\begin{array}{cc}
\mu (Z) \leq p_M (Z) & \textnormal{ for all } Z\in E.
\end{array}
\]
 
Since $M$ is a neighborhood of $0\in E$, by Proposition \ref{closureGauge}, the gauge function $p_M$ is continuous on $E$. Hence $\mu$ is a continuous
$L^0$-linear function defined on $E$.

Furthermore, we have that
\[
\underset{Z \in M_1} \esssup \mu (Z) \leq \underset{Z \in M} \esssup \mu (Z) \leq
\]
\[
\begin{array}{cc}
\leq \underset{Z \in M} \esssup p_M (Z) \leq 1 < p_M (X)=\mu(X)  & \textnormal{ on }C. 
\end{array}
\]
Therefore, $X$ cannot be a weak accumulation point of $M_1$ contrary to the hypothesis of $X_\gamma$ converging weakly to $X$.
\end{proof}

We have the following corollaries:

\begin{cor}
\label{prop: closedw}
Let $(E,\left\Vert \cdot\right\Vert)$ be a $L^0$-normed module whose sum preserves the relative countable concatenation property, and let $K\subset E$  be $L^0$-convex and with the relative countable concatenation property, we have that the closure in norm coincides with the closure in the weak  topology, i.e. $\overline{K}^{\Vert \cdot \Vert} = \overline{K}^{\sigma(E,E^*)}$.
\end{cor}

Then, from now on, for any subset $K$ which is $L^0$-convex and with the relative countable concatenation property, we will denote the topological closure by $\overline{K}$ without specifying whether the topology is either weak or strong.

Let us recall some notions:

\begin{defn}
Let $E[\tau]$ be a topological $L^0$-module. A function $f:E\rightarrow\bar{L}^{0}$ is called proper if $f(E)\cap L^0 \neq \emptyset$ and $f>-\infty$. It is said to be $L^0$-convex if $f (Y X_1 + (1 − Y)X_2) \leq Y f (X_1) + (1 − Y)f (X_2)$ for all $X_1,X_2 \in E$ and
$Y \in L^0$ with $0 \leq Y \leq 1$. It said to have the local property if $1_A f(X)=1_A f(1_A X)$ for $A\in\mathcal{F}^+$ and $X\in E$. Finally, $f$ is called lower semicontinuous if the level set $V\left(Y_{0}\right)=\left\{ X\in E;\: f\left(X\right)\leq Y_{0}\right\}$ is closed for all $Y_0\in L^0$.
\end{defn}

\begin{cor}
\label{cor: lowerSemCont}
Let $(E,\left\Vert \cdot\right\Vert)$ be a $L^0$-normed module whose sum preserves the relative countable concatenation property, and let $f:E\rightarrow \bar{L}^0$  be a proper $L^0$-convex function. If $f$ is continuous, then $f$ is lower semicontinuous with the weak topology. 
\end{cor}

\begin{proof}
It is a known fact that, if $f$ is $L^0$-convex, then it has the local property (see \cite[Theorem 3.2]{key-8}).

Being $f$ $L^0$-convex and with the local property, we have that $V(Y)$ is $L^0$-convex and has the relative countable concatenation property.

Since $f$ is continuous, we have that $V(Y)$ is closed, and due to  Corollary \ref{prop: closedw}, it is weakly closed as well.

\end{proof}

Finally, we will provide an example showing that in the version of Mazur's lemma proved, rather than just take $X_{\varepsilon}$ into the $L^0$-convex hull, we must take it into the countable concatenation closure of the $L^0$-convex hull. Namely, we shall give an example of a net weakly convergent to some limit, which is not a cluster point of the $L^0$-convex hull of that net.   

\begin{example}
Given  $\Omega=(0,1)$, $\mathcal{E}=\mathcal{B}(\Omega)$ the $\sigma$-algebra of Borel, $A_{n}=[\frac{1}{2^n},\frac{1}{2^{n-1}})$ with $n\in\mathbb{N}$ and $\PP$ the Lebesgue measure. 
~\\

We define
\[
\mathcal{F}:=\sigma(\left\{A_n ; \: n\in \N \right\})
\]
the $\sigma$-algebra generated by $\left\{A_n ; \: n\in \N \right\}$.
~\\

Then, we take the $L^0(\mathcal{F})$-module
\[
L_{\mathcal{F}}^{2}\left(\mathcal{E}\right):=L^{0}\left(\mathcal{F}\right)L^{2}\left(\mathcal{E}\right)
\] 
and the $L^0(\mathcal{F})$-seminorm
\[
\left\Vert X \mid \mathcal{F} \right\Vert_{2}:=\E\left[\left|X\right|^{2}|\mathcal{F}\right]^{1/2} 
\]
as we can see defined in \cite{key-3}.
~\\

Then the following holds
\begin{equation}
\label{eq1}
L^{0}\left(\mathcal{F}\right)=\left\{\sum_{n\in \N} \alpha_n 1_{A_n} ; \: \alpha_n \in \R \right\}
\end{equation}
\begin{equation}
\label{eq2}
L_{\mathcal{F}}^{2}\left(\mathcal{E}\right)=\left\{\sum_{n\in \N} X_n 1_{A_n} ; \: X_n \in L^{2}\left(\mathcal{E}\right) \right\}
\end{equation}
\begin{equation}
\label{eq3}
\begin{array}{cc}
\left\Vert X \mid \mathcal{F} \right\Vert_{2}^2:=\sum_{n\in \N} \frac{\Vert X 1_{A_n}\Vert_2^2}{1/2^n}1_{A_n}
 & \textnormal{ for } X\in L_{\mathcal{F}}^{2}\left(\mathcal{E}\right).
\end{array}
\end{equation}

Now, we will define a net in $L_{\mathcal{F}}^{2}\left(\mathcal{E}\right)$ indexed with the set $\N ^ \N$. Given $\{ n_k\}_{k\in\N}$ we define
\[
\begin{array}{cc}
X_{\{ n_k\}_{k\in\N}} (t):=\sum_{k\in\N} 1_{A_k} \sgn [\sin 2\pi (2^{k+n_k} t - 1)] & \textnormal{ for } t\in (0,1).
\end{array}
\]

We shall show that this net converges weakly to $0$ and that $0$ is not a cluster point of $co_{L^0}\left\{X_{\{n_k\}} ; \: \{n_k\}\in \N^\N \right\}$.
~\\

Indeed, by \cite{key-12} or \cite{key-10}, we know that for each $X^* \in E^*$, there exists $Y\in L_{\mathcal{F}}^{2}\left(\mathcal{E}\right)$ such that 
\[
\left\langle X, X^* \right\rangle=\E\left[X{Y}|\mathcal{F}\right]
\]

But, for every $Y\in L_{\mathcal{F}}^{2}\left(\mathcal{E}\right)$ 
\[
\left| \E\left[X_{\{n_k\}}{Y}|\mathcal{F}\right] \right|=\underset{k\in\N}\sum{\frac{\E[1_{A_k}{Y}X_{n_k}]}{\PP(A_k)}1_{A_k}}=
\]

\[
=\underset{k\in\N}\sum{\frac{\left| \int_{1/2^k}^{1/2^{k-1}}{{Y}\sgn[\sin {2\pi (2^{k+{n_k}} s - 1  )}]} d s \right|}{1/{2^k}}1_{[\frac{1}{2^k},\frac{1}{2^{k-1}})}},
\]

and it can be proved that the later converges to $0$ on $A_k$ for $k=1,2,...$. Hence, since $Y$ is arbitrary we conclude that $X_{\{n_k\}}$ converges weakly to $0$.
~\\

On the other hand, let us see that $0$ is not a cluster point of $co_{L^0}\left\{X_{\{n_k\}} ; \: \{n_k\}\in \N^\N \right\}$.
~\\

Indeed, given $Y\in co_{L^0}\left\{X_{\{n_k\}} ; \: \{n_k\}\in \N^\N \right\}$. We have that $Y$ will be as follows 
\[
Y=\sum_{k\in \N} 1_{A_k} \sum_{i=1}^N{\alpha_k^i \sgn[\sin {2\pi (2^{k+{n_k}} s - 1  )}]} 
\]

with $N\in \N$, $\alpha_k^i\in\R$ and with $\sum_{i=1}^N{\alpha_k^i}=1$ for all $k\in\N$.

In addition, we have that it can be proved that
\[
\left\Vert 1_{A_k} \sum_{i=1}^N{\alpha_k^i \sgn[\sin {2\pi (2^{k+{n_k}} s - 1  )}]} \right\Vert_2^2 \geq 
\]
\[
\geq \PP\left(A_k \cap \bigcap_{j=1}^N \left({\sgn[\sin {2\pi (2^{k+{n_k}} s - 1  )}]=1}\right) \right) \geq \frac{1}{2^{N+K-1}}.
\]

Therefore, by using \ref{eq3}
\[
\left\Vert Y \mid \mathcal{F} \right\Vert_{2}^2 = \sum_{k\in \N}{\frac{\left\Vert 1_{A_k} \sum_{i=1}^N{\alpha_k^i \sgn[\sin {2\pi (2^{k+{n_k}} s - 1  )}]} \right\Vert_2^2}{1/2^k}1_{A_k}}\geq
\]
\[
=\sum_{k\in \N} \frac{1/2^{N+k-1}}{1/2^k} 1_{A_k}= \sum_{k\in \N} \frac{1}{2^{N-1}} 1_{A_k}.
\]

But, taking $\varepsilon:=\sum_{k\in \N} \frac{1}{2^{k}} 1_{A_k} \in L^0_{++}(\mathcal{F})$ it is clear that for each $Y\in co_{L^0}\left\{X_{\{n_k\}} ; \: \{n_k\}\in \N^\N \right\}$ there exists $A\in\mathcal{F}$ with $\PP(A)>0$ such that
\[
\begin{array}{cc}
\left\Vert Y \mid \mathcal{F} \right\Vert_{2}^2 > \varepsilon & \textnormal{ on } A.
\end{array}
\]
Hence, $0$ cannot be a cluster point of $co_{L^0}\left\{X_{\{n_k\}} ; \: \{n_k\}\in \N^\N \right\}$ as could be expected considering the classical Mazur's lemma. 
\end{example}

\end{document}